\documentclass[paper=a4, fontsize=11pt]{scrartcl} 

\usepackage[T1]{fontenc} 
\usepackage{fourier} 
\usepackage[english]{babel} 
\usepackage{amsmath,amsfonts,amsthm} 
\usepackage{amsmath, calligra, mathrsfs}
\usepackage{amssymb}
\usepackage[mathscr]{euscript}
\usepackage{verbatim}
\usepackage[dvipsnames]{xcolor}
\usepackage{mdframed} 
\usepackage{soulutf8}
\usepackage{stmaryrd}
\usepackage{tikz-cd}
\usepackage{hyperref}
\usepackage{lipsum} 

\usepackage{sectsty} 
\allsectionsfont{\centering \normalfont\scshape} 

\usepackage{fancyhdr} 
\usepackage{xurl}
\newcommand*{\sheafhom}{\mathcal{H}\kern -.5pt om}
\numberwithin{equation}{section} 
\numberwithin{figure}{section} 
\numberwithin{table}{section} 

\newtheorem{thm}{Theorem}[section]
\newtheorem{cor}[thm]{Corollary}
\newtheorem{prop}[thm]{Proposition}
\newtheorem{obs}[thm]{Observation}
\newtheorem{lem}[thm]{Lemma}

\newtheorem{conj}[thm]{Conjecture}
\newtheorem{quest}[thm]{Question}
\theoremstyle{definition}
\newtheorem{defn}[thm]{Definition}

\newtheorem{exmp}[thm]{Example}

\theoremstyle{remark}
\newtheorem{rem}[thm]{Remark}

\DeclareMathOperator{\Span}{Span}

\DeclareMathOperator{\rk}{rank}

\newcommand{\horrule}[1]{\rule{\linewidth}{#1}} 

\title{	
	\normalfont \normalsize 
	\textsc{} \\ [25pt] 
	\horrule{0.5pt} \\[0.4cm] 
	\huge Matroids satisfying the matroidal Cayley--Bacharach property and ranks of covering flats

	\horrule{2pt} \\[0.5cm] 
}

\author{Soohyun Park} 

\date{\normalsize May 6, 2022} 

\begin{document}

	\maketitle 
	
	\begin{abstract}
		\noindent Let $M$ be a matroid satisfying a matroidal analogue of the Cayley--Bacharach condition. Given a number $k \ge 2$, we show that there is no nontrivial bound on ranks of a $k$-tuple of flats covering the underlying set of $M$. This addresses a question of Levinson--Ullery motivated by earlier results which show that bounding the number of points satisfying the Cayley--Bacharach condition forces them to lie on low-dimensional linear subspaces. We also explore the general question what matroids satisfy the matroidal Cayley--Bacharach condition of a given degree and its relation to the geometry of generalized permutohedra and graphic matroids. 
	\end{abstract}
	
	\section{Introduction}
	 
	A finite subset $\Gamma \subset \mathbb{P}^n$ satisfies the Cayley--Bacharach condition of degree $r$ if a homogeneous polynomial of degree $r$ vanishing on all but one point of $\Gamma$ vanishes on all of $\Gamma$. In recent work, Levinson--Ullery \cite{LU} show that a finite subset $\Gamma \subset \mathbb{P}^n$ satisfying the Cayley--Bacharach condition of degree $r$ is covered by low-dimensional linear subspaces if $|\Gamma|$ is not very large compared to $r$ (Theorem 1.3 on p. 2 of \cite{LU}). The result was motivated by constructions relating to degrees of irrationality of smooth complete intersections. \\  
	
	More specifically, the varieties $X$ considered as motivating examples are those with generically finite dominant rational maps $X \dashrightarrow \mathbb{P}^n$ connected to (a generalization of) the Cayley--Bacharach property and special configurations of points. It is known that the generic fiber of the rational map $X \dashrightarrow \mathbb{P}^n$ satisfies the Cayley--Bacharach property with respect to the linear system $|K_X|$ (replacing homogeneous degree $r$ polynomials by sections of $K_X$). If $K_X$ is sufficiently positive, then the fibers also lie in special positions. For example, a result of Bastianelli--Cortini--De Poi (Theorem 1.1 on p. 2 of \cite{LU}) states that a finite subset $\Gamma \subset \mathbb{P}^n$ satisfying the degree $r$ Cayley--Bacharach property of degree $r$ such that $|\Gamma| \le 2r + 1$ lies on a line. The results of Levinson--Ullery (Theorem 1.3 on p. 2 of \cite{LU}) are analogues which show that $\Gamma$ still lies on a \emph{union} of low-dimensional linear subspaces when we impose a weaker linear upper bound in $r$ on the size of $\Gamma$. They are part of a more general conjectured statement which is listed below along with the result. \\

	\begin{conj} \label{lincov} (Levinson--Ullery, Conjecture 1.2 on p. 2 of \cite{LU}) \\
		Let $\Gamma \subset \mathbb{P}^n$ be a finite set of points satisfying $CB(v)$. If $|\Gamma| \le (d + 1)v + 1$, then $\Gamma$ can be covered by a union of positive--dimensional linear subspaces $P_1 \cup \cdots \cup P_k$ such that $\sum_{i = 1}^k \dim P_i = d$. \\
	\end{conj}
	
	\begin{thm} \label{levullori} (Levinson--Ullery, Theorem 1.3 on p. 2 of \cite{LU}) \\
		Conjecture \ref{lincov} holds in the following cases:
		
		\begin{enumerate}
			\item For all $v \le 2$ and all $d$. Moreover, we can take $k = 1$.
			
			\item For all $v$ and for $d \le 3$. Moreover, we may take $k \le 2$.
			
			\item For $d = 4$ and $v = 3$. Moreover, we may take $k \le 2$.
		\end{enumerate}
	\end{thm}

	Since many of the arguments used in the result of Levinson--Ullery (Theorem 1.3 on p. 2 of \cite{LU}) are combinatorial in a way that can sometimes be rewritten in terms of matroid theory (p. 14 of \cite{LU}), the authors define a matroid-theoretic analogue of the Cayley--Bacharach property.  \\
	
	\begin{defn} (p. 14 of \cite{LU}) \\
		A matroid $M$ with underlying set $E$ satisfies the \textbf{matroidal Cayley--Bacharach property of degree $a$ (denoted $MCB(a)$)} if, whenever a union of $a$ flats contains all but one point of $E$, the union contains the last point. In other words: \[ \bigcup_{i = 1}^a F_i \supset E \setminus p \Longrightarrow \bigcup_{i = 1}^a F_i = E \]  for any $p \in E$ and any flats $F_1, \ldots, F_a$ of $M$. We will work with matroids $M$ such that all falts of rank 1 have size 1 since flats of rank 1 correspond to a single point in $\mathbb{P}^n$.
	\end{defn}
	
	Note that the finite sets satisfying the Cayley--Bacharach property are represented by the underlying (finite) set of the matroid and the flats are analogous to linear subspaces. \\
		
	Using these flats of matroids, the authors ask whether an analogue of their main result holds for the matroidal Cayley--Bacharach property. We will discuss this question and a variant.
	
	\pagebreak 
	
	\begin{quest} \label{flatcover} ~\\
		\vspace{-3mm}
		\begin{enumerate}
			\item (Levinson--Ullery, Question 7.6 on p. 14 of \cite{LU})  Does the statement of Conjecture \ref{lincov} with $CB(v)$ replaced by $MCB(v)$ and dimensions of linear subspaces replaced by ranks of flats hold? Here is a more explicit statement: \\  
			
			Let $M$ be a matroid with underlying set  $E$ such that all flats of rank $1$ have size $1$. Suppose that $M$ satisfies $MCB(a)$ and  $|E| \le (d + 1)a + 1$. Let $d_i = r_i - 1$ if $r_i \ge 2$ and $d_i = 1$ if $r_i = 1$.  Is it possible to cover $M$ by a union of (possibly improper) flats $\bigcup_i F_i$ of ranks $r_i$ respectively such that $\sum_i d_i \le d$?
			
			\item We can consider a variant of the question in Part 1 since the original source refers to covering matroids by a union of flats of ``specified dimensions''. Let $M$ be a matroid of rank $r$ with underlying set $E$ of size $n$. Fix a positive integer $N = N(M)$. Suppose that $M$ satisfies $MCB(a)$. Must $M$ (meaning the underlying set $E$) be covered by a union of $\le N$ proper flats where at least one of the flats has rank $\le r - 2$?

		\end{enumerate}
	\end{quest}
	
	\begin{rem} ~\\
		\vspace{-3mm}
		\begin{enumerate}
			\item In Part 1, we replace ``dimensions'' $d_i$ in the original statement of Question 7.6 on p. 14 of \cite{LU} with $r_i - 1$ if $r_i \ge 2$, where $r_i$ is the rank of a flat $F_i$. If $r_i = 1$, we will take $d_i = 1$.  This is because the analogous geometric condition considers dimensions of spans of points in projective space and ``dimension'' does not seems to be a standard term for flats of a matroid unless we are discussing representable matroids. In the latter setting, the rank is equal to the dimension of the linear subspace spanned by the vectors corresponding to the points of the flat. Also, we consider both an interpretation of the problem using ranks of individual flats (for ``flats of specified ranks'' for Theorem \ref{nobdmcb}) and a direct analogue of Conjecture \ref{lincov} (Theorem \ref{mcbdimnobd}). \\
			
			\item The bounds on sizes of finite set $\Gamma$ (modeled by $E$ above) satisfying the (geometric) Cayely--Bacharach property in Theorem 1.1, Conjecture 1.2, and Theorem 1.3 on p. 2 of \cite{LU} are on the size of the finite set (analogous to $n = |E|$) relative to the degree (given by $a$ above). The dimension of a plane configuration is the sum of the dimension of the linear subspaces used to cover the finite set $\Gamma$ on p. 2 of \cite{LU}. \\
			
			\item For each of the questions above, we find a counterexample using a matroid satisfying $MCB(a)$ where the flats involved in the definition of $MCB(a)$ must be hyperplanes (i.e. maximal proper flats). \\
		\end{enumerate}
	\end{rem}
	
	In Section \ref{mcbans}, we find some examples of nontrivial flats satisfying the matroidal Cayley--Bacharach condition whose nontrivial covers by flats only use hyperplanes (i.e. maximal proper flats) (Example \ref{pavexmp}) when $N \le a$. Since this would mean taking $b_i = r - 1$ for all $i$ in Question \ref{flatcover}, this means that there is no nontrivial bound on the ranks of flats covering the ground set of a matroid satsifying the degree $a$ matroidal Cayley--Bacharach condition (Theorem \ref{nobdmcb}).  \\

	\begin{thm} (Theorem \ref{nobdmcb}) \\
		\vspace{-3mm}
		Take an even number $B \ge 2m + 2$ with $B | n$ and $\frac{n}{B} < B$.  Fix $k, a \le \frac{n}{B} + \frac{B}{m} - 3$. 
		
		\begin{enumerate}
			\item Let $M$ be a matroid of rank $m + 1$ satisfying $MCB(a)$ with underlying set $E$ of size $n$. Suppose that $F_1, \ldots, F_k$ is a $k$-tuple of flats covering $E$ with each proper flat of size at most $B$. There is no covering by $\le k$ flats where at least one of the flats has rank $\le r - 2$.  In other words, it is possible for all the flats $F_i$ to be hyperplanes. Question \ref{flatcover} contains an explanation of why this indicates that there is no ``nontrivial'' bound. \\ 
			
			\item In fact, the matroid from the proof of Part 1 gives a negative answer to Question \ref{flatcover} using the case $a \le \frac{n}{B} + \frac{B}{m} - 3$. More specifically, there is no upper bound on the ranks of a collection of $\le a$ flats which cover the underlying set $E$ of a matroid of rank $r$ satisfying $MCB(a)$. \\   
		\end{enumerate}
	\end{thm}
	
	\begin{rem}
		The counterexamples we study have some recursive properties regarding the matroidal Cayley--Bacharach property and \emph{some} upper bound is required in order for $MCB(a)$ to be satisfied (Proposition \ref{pavmcbcrit}). \\
	\end{rem}
	
	The example above also implies a direct translation of Conjecture 1.2 on p. 2 of \cite{LU} does \emph{not} hold.

	\begin{thm} (Theorem \ref{mcbdimnobd}) \\
		There is a matroid $M$ satisfying $MCB(a)$ with ground set $E$ ($n:= |E|$) such that there is some $d$ such that $n \le (d + 1)a + 1$ but $E$ \emph{cannot} be covered by a union of flats of total rank $d$. In other words, we have that $\bigcup_{i = 1}^k F_i = E \Longrightarrow \sum_{i = 1}^k \rk F_i \ge d + 1$. 
	\end{thm}
	
	Afterwards, we study general properties of matroids satisfying the matroidal Cayley--Bacharach condition in Section \ref{mcbpolyt}. The main tool used here is the matroid polytope determined by the basis elements. This gives a characterization of ``generic'' matroids that have appropriate connectivity properties (Theorem \ref{defmcb}).  \\
	
	\begin{thm} (Theorem \ref{defmcb}) \\
		Suppose that $M = \sum_{I \subset[n]} y_I \Delta_I$ for some $y_I \ge 0$. \\
		
		Then, the existence of a matroid $N$ satsifying the following conditions can be checked using a set-theoretic condition involving  $(n - 1)$-element subsets of $[n]$ or the sets $I$:
		
		\begin{itemize}
			\item $N$ has an underlying set of the same size $n = |E|$ such that the flats inducing facets of $P_N$ satisfy the conditions $MCB(a)$ 
			
			\item $P_N$ and the matroid polytope of $P_M$ are nondegenerate deformations of each other (i.e. those not passing through vertices)
		\end{itemize}
	\end{thm}
	
	Under additional assumptions on the collection of subsets $I \subset [n]$ such that $y_I > 0$ and connectedness-related properties of $I$, we can show that checking whether the matroid $M$ itself satsifies the matroidal Cayley--Bacharach property is equivalent to checking whether the set-theoretic analogue holds for the subsets $I$ considered (Part 2 of Theorem \ref{connsetmcbequiv}). Note that the terms below are defined in Section \ref{mcbpolyt} (Definition \ref{connstrdef}, Definition \ref{buildingdefs}, Definition \ref{smcb}). \\
	
	\begin{thm} (Theorem \ref{connsetmcbequiv}) \\
		Suppose that $M$ is a connected matroid satisfying the following conditions:
		
		\begin{itemize}
			\item $M[F, G]$ is connected for all flats $F, G$ such that $F \subset G$ or every flat $A$ of $M$ is both connected and coconnected. For example, consider the graphic matroid $M(K_n)$ of spanning trees in the complete graph $K_n$ (Remark 5.4 on p. 459 of \cite{FS}). 
			
			\item  $P_M = \sum_{I \subset [n]} y_I \Delta_I$ for some $y_I \ge 0$ such that $y_{[n]} > 0$.  As mentioned in Observation \ref{genmink}, the condition here is really one on the ranks of the flats since $z_I = \sum_{J \subset I} y_J$ and $y_I = \sum_{J \subset I} (-1)^{|I| - |J|} z_J$, where $z_I = r - \rk(\Span I)$ with $r = \rk M$ and $\Span I$ being the smallest flat containing the elements of $I$ (Proposition 2.2 and Proposition 2.3 on p. 843 of \cite{ABD}). 
		\end{itemize}
		
		Let $B$ be the collection of subsets $I \subset [n]$ such that $y_I > 0$. Then, the following statements hold:
		
		\begin{enumerate}
			\item The matroid $M$ satisfies the matroidal Cayley--Bacharach property $MCB(a)$ if and only if the set-theoretic analogue of $MCB(a)$ is satisfied by the elements of building closure $\widehat{B}$ of $B$. By the ``set-theoretic analogue'', we mean the matroidal Cayley--Bacharach condition holds with the flats replaced by elements of the building closure (Definition \ref{smcb}). 
			
			\item If $B$ is a building set, then the matroid $M$ satisfies the matroidal Cayley--Bacharach property $MCB(a)$ if and only if the set-theoretic analogue of $MCB(a)$ (Definition \ref{smcb}) is satisfied by the subsets $I \subset [n]$ such that $y_I > 0$.  
		\end{enumerate}
		
	\end{thm}

	We end with some constructions which use (directed) graphs to (recursively) determine what the sets involved would look like (Proposition \ref{recsmcb}, Proposition \ref{addsmcb}, Proposition \ref{dirgraph}). 
	
	\section*{Acknowledgements}
	
	I am very thankful to my advisor Benson Farb for his guidance and encouragement. Also, I would like to thank Benson Farb and Amie Wilkinson for extensive comments on an earlier draft.
	
	\section{Ranks of flats covering a matroid satisfying $MCB(a)$} \label{mcbans}
	
	Given a fixed positive integer $a$, we show that there are no nontrivial bounds on the dimensions of $a$ proper flats covering a matroid satisfying $MCB(a)$. Let $n = |E|$ and $r = \rk M$. The ``worst'' possible situation is when the collection of flats considered must be hyperplanes, which are the flats of rank $r - 1$. These occur when we consider paving matroids with appropriate initial parameters. Recall that a paving matroid is one where any set of size $\le r - 1$ is both independent and closed. In other words, a dependent set (equivalently a circuit since considering lower bound) must have size $\ge r$. \\

	Since \emph{any} subset of size $\le r - 1$ is also a flat, we want to eliminate these from consideration since having $F_i$ equal to such a flat would automatically imply that $M$ does \emph{not} satisfy $MCB(a)$ since we can use repeated copies of the same flat. Given an upper bound $B$ on the size of the hyperplanes, we can find a condition which implies that flats $F_i$ such that $\bigcup_{i = 1}^a F_i \supset E \setminus p$ for some $p$ must have $|F_i| \ge r$. Note that this is really a condition on size of finite sets and doesn't have anything to do with the matroid structure.
	
	\begin{lem} \label{lowboundrk}
		Let $F_1, \ldots, F_a \subset E$ be a collection of subsets of $E$ with $|F_i| \le B$ for each $1 \le i \le a$. \\
		
		If $n - 1 - B(a - 1) \ge r$, then \[ \left| \bigcup_{i = 1}^a F_i \right| \ge n - 1 \Longrightarrow |F_i| \ge r \text{ for each $1 \le i \le a$.}  \]
		
		If the $F_i$ are proper flats of a paving matroid $M$ with underlying set $E$ and rank $r$, this implies that the $F_i$ considered must be hyperplanes $M$ (i.e. flats of rank $r - 1$).
		
	\end{lem}
	
	\begin{proof}

		Suppose that $|F_1| \le r - 1$. Then, we have that 
		
		\begin{align*}
			n - 1 &\le \left| \bigcup_{i = 1}^a F_i \right| \\
			&\le |F_1| + \left| \bigcup_{i = 2}^a F_i \right| \\
			&\le r - 1 + (a - 1)B,
		\end{align*}
		
		which contradicts the assumption that $n - 1 - B(a - 1) \ge r$.
		
		\color{black}

	\end{proof}
	
	Before we study the covering question, we give an example of a matroid satisfying the degree $a$ matroidal Cayley--Bacharach condition $MCB(a)$ which doesn't have a nontrivial bound on ranks of flats covering the underlying finite set $E$. The construction we will use involves paving matroids, which are defined below.
	
	\begin{defn} (p. 24 of \cite{Ox}) \\
		A matroid $M$ is \textbf{paving} if it has no circuits of size $\le \rk M$. In particular, flats of rank $\le r - 2$ are always independent sets, where $r = \rk M$.
	\end{defn}
	
	\begin{defn} (p. 71 of \cite{Ox}) \\
		Let $k$ and $m$ be integers with $k > 1$ and $m > 0$. Suppose that $\mathcal{T}$ is a collection $\{ T_1, \ldots, T_k \}$ of subsets of a set $E$ such that each member of $\mathcal{T}$ has $\ge m$ elements, and each $m$-element subset of $E$ is contained in a unique member of $\mathcal{T}$. Such a set is called an \textbf{$m$-partition of $E$.}
	\end{defn}
	
	\begin{prop} \label{pavhyp} (Proposition 2.1.24 on p. 71 of \cite{Ox}) \\
		If $\mathcal{T}$ is an $m$-partition $\{ T_1, \ldots, T_k \}$ of a set $E$, then $\mathcal{T}$ is the set of hyperplanes of a paving matroid of rank $m + 1$ on $E$. Moreover, for $r \ge 2$, the set of hyperplanes of every $\rk r$ paving matroid on $E$ is an $(r - 1)$-partition of $E$. 
	\end{prop}
	
	We can show that the first statement of Definition \ref{flatcover} does \emph{not} hold (especially if we want a small number of linear subspaces) (Theorem \ref{nobdmcb}). Note that \emph{some} bound on the number of linear subspaces involved is needed since we can always end up with some collection of lines or planes if we use a sufficient number of flats in the cover. The example which we used (paving matroids with appropriate parameters) can be used to show that the first part of the question (i.e. the direct matroid-theoretic translation of Conjecture \ref{lincov} in Part 2 of Question \ref{flatcover}) also does \emph{not} hold. \\
	
	\begin{thm} \label{mcbdimnobd}
		There is a matroid $M$ satisfying $MCB(a)$ with ground set $E$ ($n:= |E|$) such that there is some $d$ such that $n \le (d + 1)a + 1$ but $E$ \emph{cannot} be covered by a union of flats of total rank $d$. In other words, we have that $\bigcup_{i = 1}^k F_i = E \Longrightarrow \sum_{i = 1}^k \rk F_i \ge d + 1$. 
	\end{thm}
	
	\begin{proof}

		We will take $r_i \rk F_i \ge 2$ for each flat $F_i$. In this context, Part 1 of Question \ref{flatcover} can be rephrased as whether we can keep $\sum_{i = 1}^a d_i = \sum_{i = 1}^a r_i - a \le d$. To construct a counterexample, it suffices to produce an $M$ satisfying $MCB(a)$ such that $n \le (d + 1)a + 1$ and \emph{any} cover $\bigcup_{i = 1}^a F_i = E$ has $\sum_{i  =1}^a r_i > d + a$. Note that such an example suffices when we take some of the $F_i$ to have rank $1$ (i.e. that $r_i = 1$ for some $i$) since the required lower bound for $\sum_{i = 1}^a r_i$ only get smaller. \\
		
		We will construct a paving matroid $M$ satisfying $MCB(a)$ with $n \le (d + 1) a + 1$ such that any cover $\bigcup_{i = 1}^a F_i = E$ has $\sum_{i = 1}^a r_i > d + a$. Let $m + 1$ be the rank of the paving matroid. By Proposition 2.1.24 on p. 71 of \cite{Ox}, the hyperplanes are given by elements of $m$-partitions of the ground set $E = [n]$. Since any set of size $\le m - 1$ is closed, any flats $F_i$ involved in the $MCB(a)$ definition must be hyperplanes. In this setting, the condition $\sum_{i = 1}^a r_i > d + a$ can be rewritten as $am > d + a \Longleftrightarrow d < (m - 1)a$. The condition $n \le (d + 1)a + 1$ is equivalent to having $(d + 1)a \ge n - 1 \Longleftrightarrow \frac{n - 1}{a}$. Then, having both $n \le (d + 1)a + 1$ and $\sum_{i = 1}^a r_i > d + a$ simultaneously is equivalent to having $\frac{n - 1}{a} < d < (m - 1)a$. The existence of a $d$ such that this is true is equivalent to having $\frac{n - 1}{a} < (m - 1) a \Longleftrightarrow \frac{n - 1}{a^2} < m - 1$. \\
		
		The arguments above are on the possible initial parameters. We still need to show that there is a paving matroid of rank $m + 1$ with ground set $E = [n]$ satisfying $MCB(a)$ such that $\frac{n - 1}{a^2} < m - 1$. The idea is to consider a paving matroid where there is a collection of ``big'' hyperplanes (evenly) partitioning the ground set $E$ and the remaining elements of the $m$-partition being subsets of size $m$ (i.e. subsets of size $m$ with elements from at least $2$ distinct blocks/big hyperplanes). It suffices to take the big hyperplanes to have size $\frac{n}{a}$ partitioning $E = [n]$ into $a$ parts and $\frac{n}{a} \gg m$ since losing a single block means replacement by $\ge \frac{n}{am}$ small hyperplanes to fill in the resulting gap. This implies that $MCB(a)$ is satisfied since using any smaller number of big hyperplanes of size $\frac{n}{a}$. The condition $\frac{n}{a} \gg m$ is satisfied when $m = \frac{n}{a^{\frac{3}{2}}}$ when $a$ is a cube and $a^{\frac{3}{2}}$ divides $n$. It suffices to consider $n$ such that $n$ is divisible by $a^2$ and $a$ is a square. For sufficiently large $n, m, a$ it is clear that we can both have $\frac{n}{a} \gg m$ and $\frac{n - 1}{a^2} < m - 1$ if $m = \frac{n}{a^{\frac{3}{2}}}$.

	\end{proof}

	For Part 2 of Question \ref{flatcover}, we would like to find a bound on possible hyperplanes involved in the cover defining the matroidal Cayley--Bacharach property $MCB(a)$ of degree $a$. Note that Theorem \ref{mcbdimnobd} implies that there is a negative answer when $N(M) = a$ (the degree used in the matroidal Cayley--Bacharach proeprty). \\ 
	
	\begin{exmp} \label{pavexmp}
		Take an even number $B \ge 4$ with $B | n$ and $\frac{n}{B} < B$. We construct a rank $2$ paving matroid satisfying $MCB(a)$ for $a \le \min (\frac{n}{B} + \frac{B}{2} - 3, \frac{n - 3}{B} + 1)$. By Lemma \ref{lowboundrk}, the second term in the pair on the right hand side reduces the flats under consideration to hyperplanes. Note that paving matroids satisfying $MCB(a)$ must have the flats involved in covers by $\le a$ distinct flats equal to hyperplanes if the flats here are proper. This uses the following characterization of paving matroids by possible subsets of the underlying set giving rise to hyperplanes. \\

		Let $n = |E|$. We can set $m = 1$ above and make $T_1, \ldots, T_a$ equal to a subcollection of distinct subsets of $\{ 1, \ldots, n \}$ from the following families: 
			
		\begin{itemize}
			\item $\frac{n}{B}$ subsets of size $|B|$ partitioning $\{ 1, \ldots, n \}$ into $\frac{n}{B}$ parts
			
			\item $\binom{B}{2} B^2$ subsets of size $2$ consisting of pairs of points from distinct blocks of size $B$
		\end{itemize}
		
		This collection of subsets yields hyperplanes of a paving matroid $M$ satisfying $MCB(a)$. If some subcollection of these subsets is missing an element of $E$, it is missing $\ge 2$ elements. We first find collections of $a$ subsets we can use so that $a \le \frac{n}{B} + \frac{B}{2} - 3$. Since $|B| > 2$, the number of subsets used is minimized when we maximize the number of subsets of size $B$ and minimize the number of size $2$ used. Also, we use at most $\frac{n}{B} - 1$ of the subsets of size $B$ and there are at least $B$ elements of $E$ left to fill using the collection of pairs.  This would mean using $\frac{n}{B} - 1$ subsets of size $B$ and at most $\frac{B}{2} - 2$ pairs. However, this would leave us with at least $4$ missing elements. Using fewer subsets of size $B$ and more pairs would mean that we would use too many (i.e. more than $a$) subsets. Thus, the matroid $M$ is a rank $3$ paving matroid satisfying $MCB(a)$ for $a \le \frac{n}{B} + \frac{B}{2} - 3$. Examples where this procedure goes through is $n = 20$, $B = 5$, and $a = 4, 5$.
		
	\end{exmp}

	A generalization of Example \ref{pavexmp} can be used to give a negative answer to Part 1 of Question \ref{flatcover} for a fixed length $a$. In the context of the comments below Question \ref{flatcover}, the flats $F_i$ have rank $\le r - 1$. \\

	\begin{thm} \label{nobdmcb}
		\vspace{-3mm}
		Take an even number $B \ge 2m + 2$ with $B | n$ and $\frac{n}{B} < B$.  Fix $k, a \le \frac{n}{B} + \frac{B}{m} - 3$. 
		
		\begin{enumerate}
			\item Let $M$ be a matroid of rank $m + 1$ satisfying $MCB(a)$ with underlying set $E$ of size $n$. Suppose that $F_1, \ldots, F_k$ is a $k$-tuple of flats covering $E$ with each proper flat of size at most $B$. There is no nontrivial upper bound on the ranks of proper flats $F_i$ which applies to all such matroids $M$ satisfying $MCB(a)$.  In other words, it is possible for all the flats $F_i$ to be hyperplanes. Question \ref{flatcover} contains an explanation of why this indicates that there is no ``nontrivial'' bound. \\ 
			
			\item In fact, the matroid from the proof of Part 1 gives a negative answer to Question \ref{flatcover} using the case $N \le \frac{n}{B} + \frac{B}{m} - 3$. More specifically, there is no upper bound on the ranks of a collection of $\le a$ flats which cover the underlying set $E$ of a matroid of rank $r$ satisfying $MCB(a)$.  
		\end{enumerate}

	\end{thm}

	\begin{proof}
	 	\begin{enumerate}
	 		\item Let $E = \{ 1, \ldots, n \}$ be the underlying set of the matroid. The statements above follow from adapting the argument used in Example \ref{pavexmp} to subsets of size $m$ and paving matroids of rank $m + 1$ in place of subsets of size $2$ and paving matroids of rank $3$. Fix $B \ge 2m + 2$ with $\frac{n}{B} < B$ If $k, a \le \frac{n - m - 1}{B} + 1$, then Lemma \ref{lowboundrk} shows that the flats under consideration must be hyperplanes and we are done. Suppose that this is \emph{not} the case. Consider the paving matroid of rank $m + 1$ with the following subsets of size $\ge m$ as hyperplanes:
	 		
	 		\begin{itemize}
	 			\item $\frac{n}{B}$ subsets of size $|B|$ partitioning $\{ 1, \ldots, n \}$ into parts
	 			
	 			\item $A_{m, \frac{n}{B}}$ subsets of size $m$, where $A_{m, u}$ denotes the number of ordered partitions of $m$ into $B$ distinct parts with at least 2 nonempty parts. This corresponds to $m$-tuples with points with elements coming from at least $2$ different blocks of size $B$ from the first bullet.
	 		\end{itemize}
	 		
	 		We claim that this collection of subsets yield the hyperplanes of a matroid $M$ satisfying $MCB(a)$. In other words, we would like to show that a subscollection of $a$ subsets \emph{not} covering $E$ is missing $\ge 2$ elements. Note that we assumed that $a \le \frac{n}{B} + \frac{B}{m} - 3$. Since $|B| > m$, the number of subsets is minimized when we maximize the number of subsets of size $B$ and minimize the number of size $m$ subsets. For a non-covering collections of subsets, we use at most $\frac{n}{B} - 1$ of the subsets of size $B$ and there are at least $B$ elements of $E$ left to fill using the $m$-tuples of points. This would mean using $\frac{n}{B} - 1$ elements of size $B$ and at most $\frac{B}{m} - 2$ pairs. However, this would leave us with at least $2m$ missing elements. Using a smaller number of subsets of size $B$ and more $m$-elements would increase the number of subsets used by at least $\frac{B}{m}$. Thus, the matroid $M$ we obtain is a rank $m + 1$ paving matroid satisfying $MCB(a)$ for $a \le \frac{n}{B} + \frac{B}{m} - 3$. Note that the same arguments that we have just used imply that at least $\frac{n}{B} - 1$ of the blocks of size $B$ must be used.
	 		
	 		\item The same reasoning as Part 1 applies since having $\le a$ covering $E$ would require all of them to be hyperplanes. This is because any flat of rank $\le m - 2$ would have size $\le m - 2$. The argument in Part 1 implies that we need at least $\frac{n}{B} - 1$ hyperplanes of size $B$ in order to cover $E$ with $\le a$ flats. Since there are only $\le \frac{B}{m} - 2$ available flats to use for the cover, any remaining flats (even when we use hyperplanes) do not have enough elements of $E$ to cover the remaining elements of $E$ not covered by the earlier $\frac{n}{B}$ hyperplanes of size $B$.
	 	\end{enumerate}
		
	\end{proof}

	We can make some statements on which paving matroids yield hyperplanes compatible with the matroidal Cayley--Bacharach condition $MCB(a)$. They show that the restriction of the matroidal Cayley--Bacharach property to paving matroids has a recursive property and that some upper bound on $a$ is necessary in order for $MCB(a)$ to hold for a paving matroid of a given rank.

	\begin{prop} \label{pavmcbcrit} ~\\
		
		\vspace{-5mm}
		
		\begin{enumerate}

			\item  Let $M$ be a paving matroid of rank $m + 1$ with underlying set $E = \{ 1, \ldots, n \}$. Suppose that $|F| \le B$ for all flats $F$ of $M$. Fix $a \ge 3$. Suppose that $M$ satisfies $MCB(a)$. Fix a hyperplane $A$ of $M$. Let $R$ be the paving matroid on $E \setminus A$ of rank $m + 1$ with hyperplanes given by $H \cap (E \setminus A)$ for hyperplanes $H$ of $M$ containing $\ge m$ elements of $E \setminus A$. Then, $R$ satisfies $MCB(a - 1)$.

			\item Fix an integer $m \ge 3$. If $a$ is sufficiently large, there is a paving matroid $M$ of rank $m + 1$ with underlying set $E = \{ 1, \ldots, n \}$ that does \emph{not} satisfy $MCB(a)$.

		\end{enumerate}
	\end{prop}
	
	\begin{proof}
		\begin{enumerate}
			\item Suppose that $R$ does \emph{not} satisfy $MCB(a - 1)$. Then, there are flats $F_i$ of $R$ such that $\bigcup_{i = 1}^{a - 1} F_i = (E \setminus A) \setminus p$ for some $p \in E \setminus A$. By definition, there are hyperplanes $H_i$ of $M$ such that $|H_i \cap (E \setminus A)| \ge m$ and $F_i = H_i \cap (E \setminus A)$. Consider the union of flats of $M$ given by $A \cup \left( \bigcup_{i = 1}^{a - 1} H_i \right)$. Since $F_i = H_i \cap (E \setminus A)$, the only ``new'' elements added to $A$ come from those of $F_i$. This means that $A \cup \left( \bigcup_{i = 1}^{a - 1} H_i \right) = E \setminus p$ and $M$ does \emph{not} satisfy $MCB(a)$.
			
			\item Let $A$ be a subset of $E$ of size $\ge m$. The hyperplanes of a paving matroid $M$ of rank $m + 1$ with ground set $E$ with $A$ as a hyperplane split into the following categories (first, second, third cateogries):
			
				\begin{itemize}
					\item \textbf{Type 1}: The hyperplane $A$ itself
					
					\item \textbf{Type 2}: Hyperplanes containing $\ge m$ elements of $E \setminus A$
					
					\item \textbf{Type 3}: Hyperplanes containing $v$ elements of $A$ and $w$ elements of $E \setminus A$, where $1 \le v, w \le m - 1$
				\end{itemize}
			
			The last category gives the rest of the hyperplanes since each $m$-tuple of points of $E$ is contained in a \emph{unique} hyperplane. This means that we avoid repeating $m$-tuples coming from the first and second category. The conditions listed in the last category are given by this reasoning. \\
			
			In the last category, hyperplanes with $m - 1$ elements of $E \setminus A$ give rise to a partition of the subset of $A$ not used by hyperplanes in the second category. This is because $m$-tuples cannot be repeated among different hyperplanes of $M$. If a hyperplane contains $m - 2$ elements of $E \setminus A$, it contains $\ge 2$ elements of $A$ which do \emph{not} appear among the elements of the partition given by the hyperplanes of $M$ with $m - 1$ elements of $E \setminus A$. In general, hyperplanes of $M$ with $m - P$ elements of $E \setminus A$ have $\ge P$ elements of $A$ which are have not appeared in hyperplanes using more elements of $E \setminus A$. \\

			Consider a paving matroid $M$ of rank $m + 1$ with ground set $E$ containing $A$ as a hyperplane satisfying the following conditions:
			
				\begin{itemize}
					\item \textbf{Condition 1}: The Type 2 hyperplanes do \emph{not} contain any elements of $A$. In other words, suppose that hyperplanes of the second type form an $m$-partition of $E \setminus A$. 
					
					\item \textbf{Condition 2}: There is a collection of Type 3 hyperplanes (e.g. those with $m - 1$ elements of $E \setminus A$) such that the union of the elements of elements of $E \setminus A$ from each hyperplane $A$ has size $|E \setminus A| - 1$. 
				\end{itemize}
			
			Given a paving matroid $M$ satisfying both Condition 1 and Condition 2, let $\mathcal{C}$ be a collection of hyperplanes satisfying the properties listed in Condition 2. Taking the union of the hyperplanes in $\mathcal{C}$ with $A$, we obtain a union of hyperplanes of size $|E| - 1$. This implies that $M$ does \emph{not} satisfy $MCB(a)$ for $a \ge |\mathcal{C}| + 1$ since adding more hyperplanes either keeps the size of the union equal to $|E| - 1$ or makes it equal to $E$. The size is equal to $|E| - 1$ if we either keep repeating hyperplanes which were already used or only add new hyperplanes which do \emph{not} contain the element left out by the union of the hyperplanes in $\mathcal{C}$ and $A$. Thus, it suffices to show that there is a paving matroid satisfying both Condition 1 and Condition 2. \\ 
			
			As stated in the definition of Condition 1, we start by forming an $m$-partition of $E \setminus A$. Focusing on Type 3 hyperplanes with $m - 1$ elements of $E \setminus A$, we find that we need to use \emph{all} $(m - 1)$-element subsets of $E \setminus A$ in order to account for $m$-tuples in $E$ with $m - 1$ elements of $E \setminus A$ since the hyperplanes in the second category (i.e. those with $\ge m$ elements of $E \setminus A$) do not contribute any $m$-tuples containing elements of $A$. For the Type 3 hyperplanes, we take the hyperplanes to be the $m$-tuples which are \emph{not} covered by the $m$-tuples contained in a hyperplane of Type 1 (i.e. the hyperplane $A$) or one of Type 2. The resulting paving matroid satisfies Condition 2 since we can choose the collection in the definition of Condition 2 to be the $m$-tuples contained in a fixed $(|E \setminus A| - 1)$-element subset of $E \setminus A$.  
			
		\end{enumerate}
	\end{proof}

	\section{Matroids satisfying $MCB(a)$}

	This section studies families of matroids satisfying $MCB(a)$ including ``generic'' cases and those arising from graphs.
	
	\subsection{Matroid polytopes and $MCB(a)$} \label{mcbpolyt}
	
	In this section, we outline results on ``generic'' (connected) matroids satisfying the matroidal Cayley--Bacharach property $MCB(r)$ (of degree $r$) which can be determined set-theoretically (Theorem \ref{connsetmcbequiv} and Theorem \ref{genmink}). We can translate this into properties of ranks of flats that cover the underlying set of such matroids (Corollary \ref{setmcbflatcov}). Finally, we give some more concrete information on the structure of the sets involved (Proposition \ref{recsmcb} and Proposition \ref{addsmcb}). \\

	We will study matroids satisfying these properties via polytopes built out of them which are uniquely defined by the starting matroids.

	\begin{prop} (Feichtner--Sturmfels, Proposition 2.3 on p. 441 of \cite{FS}) \\
		The matroid polytope $P_M$ associated to a matroid $M$ with an underlying set $E$ of size $n$ (written as $[n] := \{ 1, \ldots, n \}$) is \[ P_M = \left\{ x \in \Delta : \sum_{i \in F} x_i \le \rk F \text{ for all flats } F \subset [n] \right\}, \] where $\Delta = n\Delta_E$. 
	\end{prop}
	
	Alternatively, this is the convex hull of vectors $e_B := \sum_{i \in B} e_i$ for bases $B$ of the matroid $M$. Note that $P_M$ is uniquely determined by $M$ (Theorem 4.1 on p. 311 of \cite{GGMS}) and that this property has even been used to define a matroid in Definition 2.1 on p. 440 of \cite{FS}. Each of these can be taken to be a \emph{signed} Minkowski sum of simplices.
	
	\begin{prop} (Ardila--Benedetti--Doker and Postnikov, Proposition 2.3 on p. 843 of \cite{ABD} and Proposition 6.3 and Remark 6.4 on p. 17 -- 18 of \cite{Pos})
		Any generalized permutohedron (e.g. matroid polytopes) has a decomposition as \textbf{signed Minkowski sums of simplices} with 
		
		\[ P_n( \{ z_I \} ) = \sum_{I \subset [n]} y_I \Delta_I, \] 
		
		where $y_I = \sum_{J \subset I} (-1)^{|I| - |J|} z_J$ for each $I \subset [n]$ and $z_I = \sum_{I \subset J} y_J$. \\ 
		
		In addition, any such Minkowski sum gives a generalized permutohedron (Proposition 2.2.3 on p. 14 of \cite{Dok}).  The latter condition is equivalent to having the $z_I$ satisfy submodular inequalities equivalent to the definition of some rank function on a matroid (Theorem 2.21 on p. 13 of \cite{Dok}). 
	\end{prop}
	
	\begin{obs} \label{genmink}
		For an open/generic/top-dimensional subset of the deformation cone parametrizing generalized permutohedra (i.e. deformations of the usual permutohedron), one we can take $y_I \ge 0$ for each $I$ (Remark 6.4 on p. 1043 of \cite{Pos}). The fact that $y_I = \sum_{J \subset I} (-1)^{|I| - |J|} z_J$ for each $I \subset [n]$ implies that the condition $y_I \ge 0$ for each $I \subset [n]$ is really an inequality on the ranks $r - z_I$ of the flats.
	\end{obs}
	
	In the setting of Observation \ref{genmink}, we can make some characterizations of matroids satisfying the matroidal Cayley-Bacharach property. When the $y_I \ge 0$ for all $I \subset [n]$. the facets have natural connections with nested sets and buildings.
	
	\begin{defn} \label{buildingdefs} (Building set and closure, Lemma 3.9 and Lemma 3.10 on p. 450 of \cite{FS}, Definition 7.1 on p. 1044 of \cite{Pos}) 
		\begin{enumerate}
			\item A collection $B$ of nonempty subsets of $[n] = \{ 1, \ldots, n \}$ is a \textbf{building set} on $[n]$ if it satisfies the following conditions:
				
				\begin{itemize}
					\item If $I, J \in B$ and $I \cap J \ne \emptyset$, then $I \cup J \in B$.
					
					\item $B$ contains all singletons $\{ i \}$ for $i \in [n]$.
				\end{itemize}
			
			\item Given a collection $\mathcal{F}$ of subsets of $[n]$, let $\widehat{\mathcal{F}}$ be the unique minimal collection containing $\mathcal{F}$ of subsets such that $\widehat{\mathcal{F}}$ is a building set on $[n]$. The collection $\widehat{\mathcal{F}}$ is called the \textbf{building closure}. Note that this exists for \emph{any} family of subsets $\mathcal{F}$ of $[n]$. 
		\end{enumerate}
	\end{defn}
	
	\vspace{2mm}
	
	These properties are connected to an alternate description of the facets of the matroid polytope $P_M$ when the generic property described in Observation \ref{genmink} holds (i.e. when $y_I \ge 0$ for all $i \subset [n]$).

	\begin{prop} (Proposition 3.12 and Corollary 3.13 on p. 151 of \cite{FS}) \label{buildfacet} \\
		Given a Minkowski sum of (scaled) simplices $\sum_{I \subset [n]} y_I \Delta_I$ for some $y_I \ge 0$, let $B$ be the collection of subsets $I \subset [n]$ such that $y_I > 0$. The polytope $\sum_{I \subset [n]} y_I \Delta_I$ consists of vectors $(x_1, \ldots, x_n) \in \mathbb{R}_{\ge 0}^n$ such that $x_1 + \ldots + x_n = |B|$ and \[ \sum_{i \in G} x_i \ge | \{ I \in B : I \subset G \}|  \] for all subsets $G \subset [n]$. \\
		
		It suffices to take subsets $G$ in the building closure $\widehat{B}$ of $B$. If $[n] \in B$, the condition that the linear form $\sum_{i \in G} x_i$ is minimized on a facet is \emph{equivalent} to $G$ being in the building closure $\widehat{B}$. 
	\end{prop}

	Under an additional connectivity assumption, this is entirely determined by set-theoretic considerations corresponding to the ranks of the flats of the given matroid.

	\begin{defn} \label{connstrdef} (p. 457 of \cite{FS}, p. 183 of \cite{Dl}) \\
		Given flats $F, G$ of $M$ with $F \subset G$, the subsets \[ M[F, G] := \{ b \cap (G \setminus F) : b \in M, |b \cap F| = \rk F, |b \cap G| = \rk G \} \] of the underlying set define a matroid with ground set $G \setminus F$. 
	\end{defn}

	We now state a structural result connecting the matroid polytope with the matroidal Cayley--Bacharach condition $MCB(a)$.

	\begin{thm} \label{connsetmcbequiv}

		Suppose that $M$ is a connected matroid satisfying the following conditions:
			
			\begin{itemize}
				\item $M[F, G]$ is connected for all flats $F, G$ such that $F \subset G$ or every flat $A$ of $M$ is both connected and coconnected. For example, consider the graphic matroid $M(K_n)$ of spanning trees in the complete graph $K_n$ (Remark 5.4 on p. 459 of \cite{FS}). 
				
				\item  $P_M = \sum_{I \subset [n]} y_I \Delta_I$ for some $y_I \ge 0$ such that $y_{[n]} > 0$.  As mentioned in Observation \ref{genmink}, the condition here is really one on the ranks of the flats since $z_I = \sum_{J \subset I} y_J$ and $y_I = \sum_{J \subset I} (-1)^{|I| - |J|} z_J$, where $z_I = r - \rk(\Span I)$ with $r = \rk M$ and $\Span I$ being the smallest flat containing the elements of $I$ (Proposition 2.2 and Proposition 2.3 on p. 843 of \cite{ABD}). 
			\end{itemize}
		
		Let $B$ be the collection of subsets $I \subset [n]$ such that $y_I > 0$. Then, the following statements hold:

		\begin{enumerate}
			\item The matroid $M$ satisfies the matroidal Cayley--Bacharach property $MCB(a)$ if and only if the set-theoretic analogue of $MCB(a)$ is satisfied by the elements of building closure $\widehat{B}$ of $B$. By the ``set-theoretic analogue'', we mean the matroidal Cayley--Bacharach condition holds with the flats replaced by elements of the building closure (Definition \ref{smcb}). 
			
			\item If $B$ is a building set, then the matroid $M$ satisfies the matroidal Cayley--Bacharach property $MCB(a)$ if and only if the set-theoretic analogue of $MCB(a)$ (Definition \ref{smcb}) is satisfied by the subsets $I \subset [n]$ such that $y_I > 0$. 
		\end{enumerate}

	\end{thm}

	\begin{rem}
		If we remove the initial (co)connectivity assumption, counterparts of Part 1 and Part 2 hold with the matroidal Cayley--Bacharach property replaced by its restriction to flats which define facets of the matroid polytope $P_M$ (``flacets''  in Proposition 2.6 on p. 443 of \cite{FS}).
	\end{rem}
	
	\begin{proof}
		Since $\widehat{B} = B$ if $B$ is a building set, Part 2 follows from Part 1. Thus, it suffices to prove Part 1. If $M[F, G]$ is connected for all flats $F \subset G$, the hyperplanes giving the boundary of the half-spaces $\sum_{i \in F} x_i \le \rk F$ for each flat $F$ each yield facets of the matroid polytope $P_M$. This is because there is an an equivalence between complexes whose vertices correspond to \emph{all} connected flats and those that yield facets of the matroid polytope respectively (Theorem 5.3 on p. 459 of \cite{FS})). Alternatively, one can assume that each flat is both connected and co-connected (Proposition 2.4 on p. 184 of \cite{Dl}). \\
		
		The observations made above imply that each of the flats of $M$ define a facet of its matroid polytope $P_M$. By Proposition \ref{buildfacet}, the decomposition into a Minkowski sum $\sum_{I \subset [n]} y_I \Delta_I$ with $y_I \ge 0$ and $y_{[n]} > 0$ implies that elements of the building closure $\widehat{B}$ correspond to facets of $P_M$. Tracing through the correspondences, we find that that the flats of $M$ are given by subsets of $[n]$ in the building closure $\widehat{B}$. Comparing the corresponding normal vectors implies that the matroidal Cayley--Bacharach condition $MCB(a)$ is equivalent to its set-theoretic counterpart applied to the elements of the building closure (Definition \ref{smcb}). 
		
	\end{proof}
	
	Even without the connectedness assumption of Theorem \ref{connsetmcbequiv}, we can still characterize ``generic'' polytopes coming from those satisfying $MCB(r)$ up to a deformations of the the matroid polytopes involved. By ``deformation'', we mean parallel translations of facets passing through the vertices (p. 1041 of \cite{Pos}). An example is shown in Figure 2 on p. 1979 of \cite{CL}. \\

	\begin{thm} \label{defmcb} 
		Suppose that $M = \sum_{I \subset[n]} y_I \Delta_I$ for some $y_I \ge 0$. \\
		
		Then, the existence of a matroid $N$ satsifying the following conditions can be checked using a set-theoretic condition involving  $(n - 1)$-element subsets of $[n]$ or the sets $I$:
		
			\begin{itemize}
				\item $N$ has an underlying set of the same size $n = |E|$ such that the flats inducing facets of $P_N$ satisfy the conditions $MCB(a)$ 
				
				\item $P_N$ and the matroid polytope of $P_M$ are nondegenerate deformations of each other (i.e. those not passing through vertices)
			\end{itemize}
	
	\end{thm}
	
	\begin{proof}
		By Proposition 2.6 on p. 1980 of \cite{CL}, it suffices to check when the matroid polytopes have the same normal fan.  In the comparisons of normal cones, note that two collections of vectors define the same cone if and only if they can be transformed to each other using weighted permutation matrices. Since $M$ is a Minkowski sum of the simplices $\Delta_I$, its normal fan is the common refinement of those of the simplices $\Delta_I$. Since the normal fans consist of cones generated by the (outer) normal vectors of the facets, it suffices to find the facets of a matroid polytope $P_N$. On the other hand, the facets of $P_N$ consist of those coming from flats of $[n]$ and those from $(n - 1)$-element subsets of the ground set $N$ (Proposition 2.3 on p. 441 of \cite{FS} and p. 930 of \cite{Ki}). The $(n - 1)$-element subsets don't affect the $MCB(a)$ condition and the only ones inducing a nontrivial condition involving the flats is the restriction of the $MCB(a)$ condition to the flats which induce facets of the matroid polytope.
		
	\end{proof}
	
	Finally, we discuss the implications of Theorem \ref{connsetmcbequiv} and Theorem \ref{defmcb} for a question of Levinson--Ullery (Question 7.6 on p. 14 of \cite{LU}) for possible ranks of flats satisfying the $MCB(r)$ property. \\

	\begin{cor} \label{setmcbflatcov}
		Under the conditions of Part 2 of Theorem \ref{connsetmcbequiv}, the possible sizes of $|I|$ from subsets $I \subset [n]$ with $y_I > 0$ in Theorem \ref{connsetmcbequiv} and Theorem \ref{defmcb} determine the possible ranks of flats covering the underlying set of a matroid $M$ in Theorem \ref{connsetmcbequiv} and the matroid $N$ we ``deform'' into in Theorem \ref{defmcb}. This essentially addresses Question 7.6 on p. 14 of \cite{LU} for the generic (connected) matroids discussed in these results. 
	\end{cor}

	We end with some comments on the sets involved.
	
	\begin{defn} \label{smcb}
		Let $E$ be a finite set of size $n$. A collection of subsets of $E$ satisfies the \textbf{set-theoretic matroidal Cayley--Bacharach property $sMCB(r)$} if $\bigcup_{i = 1}^r F_i \supset E \setminus p \Longrightarrow \bigcup_{i = 1}^r F_i = E$ for the given collection of proper subsets $F_1, \ldots, F_r$ of $E$ and $p \in E$.
	\end{defn}
	
	Note that the condition does \emph{not} impose a restriction on $r$-tuples of subsets $F_1, \ldots, F_r$ such that $|\bigcup_{i = 1}^r F_i| \le n - 2$ since it is not possible for these to contain $E \setminus p$ for any $p \in E$. We can understand possible underlying sets of subsets of $E$ satsifying $sMCB(r)$ recursively where the condition is nontrivial. This depends on a counting argument.
	
	\begin{prop} \label{recsmcb}
		The subsets of $E$ satsifying $sMCB(r)$ can be determined recursively using minimal covers of subsets of $E$.
	\end{prop}
	
	\begin{proof}
		The subsets $F_1, \ldots, F_r$ satisfying the $sMCB(r)$ property depends on the following parts:
		
		\begin{itemize}
			\item A collection of ``ambient sets'' $A \subset E$ of size $\le n - 2$ or $n$ (which will eventually be taken to be the union of $F_1, \ldots, F_r$)

			\item Subsets $F_1, \ldots, F_r \subset A$ such that $\bigcup_{i = 1}^r F_i = A$. This really depends on the number of subsets $F_i$ used in a \emph{minimal} cover of $A$ (say $m \le r$) since the remaining $r - m$ subsets can be any subsets of $A$ and still give a cover of $A$. By ``minimal'', we mean that removing any of the $F_i$ will give a collection of subsets of $A$ whose union is no longer equal to $A$. Thus, it suffices to consider the \emph{minimal} covers of $A$ by $\le r$ subsets. \\
			
			This can be constructed recursively. Let $T_{a, b}$ be the number of minimal covers of a set of size $a$ by a collection of $b$ subsets (will take $b \le r$ in this case). We can split this into cases depending on the number of elements \emph{not} covered by a collection of $b - 1$ subsets. For particular number of missing elements $r$, we set the union of the  $b - 1$ subsets equal to a particular subset of $A$ with $|A| - r$ elements. There are $\binom{|A|}{|A| - r}$ choices for such a subset. Fixing a subset $U \subset A$ with $|A| - r$ elements, we have that $F_r$ can be any subset of $A$ containing the remaining $r$ elements (giving $2^{n - r}$ choices) and the number of choices of (unordered) collections of (nonempty) subsets $F_1, \ldots, F_{r - 1}$ of $A$ whose union is equal to $U$ is $T_{|A| - r, b - 1}$. This gives us the recursive relation \[ T_{a, b} = \sum_{r = 1}^{a - 1} \binom{a}{a - r} 2^{n - r} T_{a - r, b - 1}. \] 
			
			If we know $T_{u, b - 1}$ for all $u$, then we can compute $T_{a, b}$. In other words, we can treat this as induction on the second index $b$ (eventually setting $b = r$). As base cases, we can use $T_{a, 1} = 1$ (single subset equal to $A$). If $b = 2$, this means choosing a (nonempty) subset of $A$ and having the second set contain its complement. For each $m$, there are $\binom{a}{m}$ choices of $m$-element subsets of $A$ and $2^{a - m}$ choices for subsets of $A$ containing the complement of the first subset in $A$. This means that $T_{a, 2} = \sum_{m = 1}^a \binom{a}{m} 2^{a - m}$.  \\
			
			It may also be possible to relate this to \emph{disjoint} covers by some collection of elements. After a disjoint cover, we can add whatever elements of $A$ we want to each of the subsets $F_1, \ldots, F_r$ involved (possibly adding nothing to one or more subsets). After choosing a disjoint cover, this is a matter of choosing any $r$ (possibly empty) subsets of $A$ (which gives $(2^a)^r = 2^{ar}$ choices). By Proposition 2.6 on p. 1032 -- 1033 of \cite{Pos}, the disjoint covers of $A$ by $r$ elements correspond to the $(a - r)$-dimensional faces of the permutohedron $P_a(x_1, \ldots, x_a)$ (for some choice of fixed $x_1 > \cdots > x_a$) formed by the convex hull of the points formed by permuting the coordinates of the point $(x_1, \ldots, x_a)$. 
		\end{itemize}
	\end{proof}

	\begin{prop} \label{addsmcb} ~\\
		Suppose that $M$ is a matroid such that any $r$-tuple of flats $F_1, \ldots, F_r$ satsifies the following property: For any $p$, there is an $x_p$ such that $p \notin F_i \Longrightarrow x_p \notin F_i$. \\
		
		Then, the matroid $M$ satsifies the matroid Cayley--Bacharach property $MCB(r)$. Also, the flats $F_i$ must come from path covers of some directed graph with the paths being maximal among those sharing the same starting point. The lengths of maximal paths bound the ranks of the flats involved. The structure of the graph also gives an upper bound on the number of points involved. 

	\end{prop}
	
	\begin{proof}
		A special case where $sMCB(r)$ is satisfied is the case where $p$ \emph{not} being contained in a subset $F_i \subset E$ among $F_1, \ldots, F_r$ means that there is some $x_p \in E$ such that $x_p \notin F_i$. This is equivalent to the statement that $x_p \in F_i \Longrightarrow p \in F_i$. Note that the choice of $x_p$ might not necessarily be unique. Then, we can build a directed graph with an edges $i \longrightarrow j$ if and only if we can set $i = x_j$. Since the $sMCB(r)$ condition is \emph{not} affected by situations where $|\bigcup_{i = 1}^r F_i| \le n - 2$, we will restrict ourselves to the situation where $\bigcup_{i = 1}^r F_i = E$. This means that the graphs under consideration are those that involve \emph{all} the elements of $\{ 1, \ldots, n \}$. \\
		
		Each vertex corresponds ot an element of $E$. Note that the directed graphs which arise aren't completely arbirtrary. Split the graph into connected components of the underlying undirected graph. Fix a maximal directed path going in one direction. Then, any remaining vertices (which correspond to elements of $E$) must come from paths that \emph{enter} the maximal directed path at a vertex whcih is \emph{not} an endpoint since joining the new paths at such points would contradict the maximality assumption. Given a particular possible connected graph, the subsets $F_i$ of $E$ must come from paths which keep going until we encounter a loop. In other words, we are looking for paths which are maximal among those with the same starting point. This means that $sMCB(r)$ is equivalent to determining possible covers of directed graphs by such paths. As a consequence of this construction, we find that a particular graph gives upper bounds for the ranks of a collection of $r$ flats which cover $E$. 
	\end{proof}
	
	\subsection{Special case of graphs} \label{graphmcb} 
	
	We consider the case where the sets in question are disjoint and $M$ is a graphic matroid. Since the flats of a direct sum of matroids $M_1 \oplus \cdots \oplus M_r$ with disjoint underlying sets $E_1, \ldots, E_r$ are of the form $F_1 \cup \cdots \cup F_r$ for flats $F_i$ of $M_i$ (p. 125 of \cite{Ox}), we can think about this as the case where $M_1 = \cdots = M_r = M$ for some matroid $M$. Note that the underlying set of the matroid $M_1 \cup \cdots \cup M_r$ is $E_1 \cup \cdots \cup E_r$, where $E_i$ is the underlying set of $M_i$. For $M^{\oplus r}$, this means taking the copies of the underlying set $E$ of $M$ to be disjoint from each other. When the direct sum is a graphic matroid, the flats and closures have a simple interpretation. \\
	
	A result of Lov\'asz--Recski \cite{LR} indicates when a repeated direct sum is a graphic matroid. \\
	
	\begin{thm} \label{dirsumgraph} (Theorem 2 on p. 332 of \cite{LR}) \\
		Given a matroid $M$ with underlying set $S$, we call it a \textbf{$k$-circuit} if $|S| = k r(S) + 1$ and $|T| = k r(T)$ for all $T \subset S$. A repeated matroid direct sum $M^{\oplus k}$ is a graphic matroid if and only if any two $k$-circuits of $(S, M)$ are disjoint. \\
	\end{thm}
	
	We can interpret a union of $r$ flats of a matroid as a single face $F$ of $M^{\oplus r}$. Let $E^i$ be the copy of $E$ in the $i^{\text{th}}$ copy $M_i$ of $M$. Let $A$ be a subset of $E^1 \cup \cdots \cup E^r$ such that removing the labels $i$ gives the full subset $E$. This corresponds to some disjoint union of $r$ sets $A_1, \ldots, A_r \subset E$ whose union is equal to $E$. Then, a variant of $MCB(r)$ can be phrased as the statement that $F \supset A \setminus \{ p \} \Longrightarrow F \supset A$. Since flats are the sets preserved under the closure operation, the first statement implies that $F \supset \overline{E \setminus \{ p \}}$. If $\overline{A \setminus \{ p \}} = A$, then we have the desired conclusion. Under the conditions of Theorem \ref{dirsumgraph}, the resulting direct sum matroid is a graph. Then, the condition that $\overline{A \setminus \{ p \} } = A$ is equivalent to the statement that for any $p \in A$, the endpoints of $p$ are connected by a path in $A \setminus \{ p \}$. In other words, the subgraph of $M^{\oplus r}$ induced by $A$ is $2$-connected. \\
	
	The reason why we stated that the above is a ``variant'' is that the ``distribution'' of $E$ over the different flats $F_1, \ldots, F_r$ in the $MCB(r)$ condition can vary. This means that we need to have the condition satisfied for \emph{all} possible $A$ satisfying the given condition. Also, note that the $MCB(r)$ condition is satisfied if the defining statement holds for all \emph{minimal} flats $F_1, \ldots, F_r$. Then, flats that are minimal under inclusion among those which can be used to give a union of $r$ flats covering $E \setminus \{ p \}$ for $p \in E$. Putting everything together, the observations above can be summarized as follows: \\
	
	\begin{prop} \label{dirgraph}
		Let $M$ be a graphic matroid with underlying set $E$ such that any two $r$-circuits of $M$ are disjoint and $r$-tuples of flats which are minimal among those covering single point complements $E \setminus \{ p \}$ are disjoint. Then, the degree $r$ matroidal Cayley--Bacharach condition is equivalent to the statement that the subgraph of $M^{\oplus r}$ induced by any subset $A \subset E^1 \cup \cdots \cup E^r$ giving a partition of $E$ as the disjoint union of $r$ subsets yields a $2$-connected subgraph of $M^{\oplus r}$. \\
	\end{prop}
	
	\begin{rem} ~\\
		\vspace{-3mm}
		\begin{enumerate}
			\item Since the objects used to define a matroid are often analogous to those used to define topological spaces, we can study what statements can be extended to higher dimensional objects. If we continue the assumption that the minimal $r$-covers by flats are disjoint, the exact argument above applies. If we remove this disjointness condition and only consider the matroid $M$ itself instead of disjoint sums, we need to consider unions of $r$ flats, which aren't necessarily flats. This complicates the argument above. \\
			
			\item If the definition of a matroid also defines a topological space (e.g. the case of uniform matroids $U_{n, n}$), we have that $\bigcup_{i = 1}^r F_i$ is a flat if $F_1, \ldots, F_r$ are flats. This means that $\bigcup_{i = 1}^r F_i \supset E \setminus \{ p \} \Longrightarrow \overline{\bigcup_{i  = 1}^r F_i} \supset \overline{E \setminus \{ p \} } $. Since we're working with the uniform matroid $U_{n, n}$, we have that $\overline{\bigcup_{i = 1}^r F_i} = \bigcup_{i = 1}^r F_i$ and $\bigcup_{i = 1}^r F_i \supset \overline{E   \setminus \{ p \} } $. Then, it suffices to show that $\overline{E \setminus \{ p \} } = E$. This is \emph{not} the case if and only if $\overline{E \setminus \{ p \} } = E \setminus \{ p \}$. In particular, this means that $\rk(E) = \rk ( E \setminus \{ p \} ) + 1$ and $E \setminus \{ p \}$ is a hyperplane of $M$. If $A \subset E \setminus \{ p \}$, then it cannot be a basis element of $M$ since it must be of maximal rank (i.e. has rank $\rk(E)$). This means that any basis element must be of the form $R \cup \{ p \}$ for some $R \subset E \setminus \{ p \}$. Then, we have that $\rk(R \cup \{ p \}) \le \rk(E \setminus \{ p \} ) + 1$. \\ 
		\end{enumerate}

	\end{rem}

Department of Mathematics, University of Chicago \\
5734 S. University Ave, Office: E 128 \\ Chicago, IL 60637 \\
\textcolor{white}{text} \\
Email address: \href{mailto:shpg@uchicago.edu}{shpg@uchicago.edu} 
\end{document}